\documentclass[12pt, reqno]{amsart}
\usepackage{inputenc}
\usepackage[T1]{fontenc}
\usepackage{lmodern}
\usepackage{amsmath}
\usepackage{amscd}
\usepackage{amssymb}
\usepackage{mathtools}
\usepackage{latexsym}
\usepackage{MnSymbol}
\usepackage{amsrefs}
\usepackage{nicefrac}
\usepackage{microtype}
\usepackage{color}

\usepackage{enumitem} 
\setlist[enumerate,1]{label=\textup{(\arabic*)}}

\usepackage[all]{xy}
\newdir{ >}{{}*!/-5pt/@{>}}

\usepackage[pdftitle={Nonarchimedean bornologies, cyclic homology and rigid cohomology},
pdfauthor={Guillermo Corti\~nas, Joachim Cuntz, Ralf Meyer, Georg Tamme},
pdfsubject={Mathematics}
]{hyperref}

\BibSpec{book}{%
  +{}  {\PrintPrimary}                {transition}
  +{,} { \textit}                     {title}
  +{.} { }                            {part}
  +{:} { \textit}                     {subtitle}
  +{,} { \PrintEdition}               {edition}
  +{}  { \PrintEditorsB}              {editor}
  +{,} { \PrintTranslatorsC}          {translator}
  +{,} { \PrintContributions}         {contribution}
  +{,} { }                            {series}
  +{,} { \voltext}                    {volume}
  +{,} { }                            {publisher}
  +{,} { }                            {organization}
  +{,} { }                            {address}
  +{,} { \PrintDateB}                 {date}
  +{,} { }                            {status}
  +{}  { \parenthesize}               {language}
  +{}  { \PrintTranslation}           {translation}
  +{;} { \PrintReprint}               {reprint}
  +{.} { }                            {note}
  +{.} {}                             {transition}
  +{} { \PrintDOI}                   {doi}
  +{} { available at \url}            {eprint}
  +{}  {\SentenceSpace \PrintReviews} {review}
}

\renewcommand*{\PrintDOI}[1]{\href{http://dx.doi.org/\detokenize{#1}}{doi: \detokenize{#1}}}

\newcommand{\comment}[1]{}  

\definecolor{orange}{rgb}{1,0,0}
\definecolor{cadmiumgreen}{rgb}{0, 0.78, 0.05}

\theoremstyle{plain}
\newtheorem{theorem}{Theorem}[section]
\newtheorem{lemma}[theorem]{Lemma}
\newtheorem{corollary}[theorem]{Corollary}
\newtheorem{proposition}[theorem]{Proposition}
\theoremstyle{remark}
\newtheorem{remark}[theorem]{Remark}
\theoremstyle{definition}
\newtheorem{definition}[theorem]{Definition}
\newtheorem{example}[theorem]{Example}
\numberwithin{theorem}{section}
\newcommand{\bgl}{\begin{equation}} 
\newcommand{\egl}{\end{equation}}
\newcommand{\btheo}{\begin{theorem}}
\newcommand{\etheo}{\end{theorem}}
\newcommand{\blemma}{\begin{lemma}}
\newcommand{\elemma}{\end{lemma}}
\newcommand{\bproof}{\begin{proof}}
\newcommand{\eproof}{\end{proof}}
\newcommand{\bbew}{\begin{beweis}}
\newcommand{\ebew}{\end{beweis}}
\newcommand{\bremark}{\begin{remark}}
\newcommand{\eremark}{\end{remark}}
\newcommand{\bex}{\begin{example}\em}
\newcommand{\eex}{\end{example}}
\newcommand{\bdefin}{\begin{definition}}
\newcommand{\edefin}{\end{definition}}
\newcommand{\bprop}{\begin{proposition}}
\newcommand{\eprop}{\end{proposition}}
\newcommand{\bcor}{\begin{corollary}}
\newcommand{\ecor}{\end{corollary}}
\newcommand{\bfa}{\begin{cases}} 
\newcommand{\efa}{\end{cases}}

\newcommand{\cO}{\mathcal O}

\newcommand\N{\mathbb N}

\newcommand\R{\mathbb R}
\newcommand\Z{\mathbb Z}

\newcommand{\bdd}{\mathcal B}
\newcommand{\coma}{\widehat}
\newcommand{\comb}{\overbracket[.7pt][1.4pt]}

\newcommand*{\comJ}[2]{\comb{\ul{#1}_{#2}}}
\newcommand*{\tub}[3]{\mathcal{T}_{#3}(#1,#2)}

\newcommand{\updagger}{\textup{\tiny\!\!\dagger}}

\newcommand{\defeq}{\mathrel{:=}} 

\newcommand{\triqui}{\vartriangleleft}
\newcommand{\rig}{\mathrm{rig}}
\newcommand{\Spec}{\operatorname{Spec}}

\newcommand\hotimes{\mathbin{\comb{\otimes}}}

\newcommand{\Hy}{\mathbb{H}}

\newcommand{\Sp}{\mathrm{Sp}}

\newcommand{\spec}{\mathrm{sp}}

\newcommand{\ev}{\mathrm{ev}}
\newcommand{\id}{\mathrm{id}}
\newcommand{\nb}{\nobreakdash}

\newcommand{\dvr}{V}
\newcommand{\dvgen}{\pi}
\newcommand{\dvf}{K}
\newcommand{\resf}{k}

\newcommand*{\abs}[1]{\lvert #1\rvert}
\newcommand*{\pepsilon}{\epsilon}

\newcommand{\ul}{\underline}

\newcommand{\chaindR}{\mathsf{cdR}}
\newcommand{\homdR}{\mathsf{hdR}}

\begin{document}

\title[Weak completions, bornologies and rigid cohomology]{Weak completions, bornologies \\ and rigid cohomology}
\author{Guillermo Corti\~nas}
\address{Dep. Matem\'atica-IMAS, FCEyN-UBA\\ Ciudad Universitaria Pab 1\\
1428 Buenos Aires\\ Argentina}
\email{gcorti@dm.uba.ar}\urladdr{http://mate.dm.uba.ar/\~{}gcorti}

\author{Joachim Cuntz}
\address{Mathematisches Institut\\
  Westf\"alische Wilhelms-Universit\"at M\"unster\\
  Ein\-stein\-str.\ 62\\
  48149 M\"unster\\
  Germany}
\email{cuntz@math.uni-muenster.de}

\author{Ralf Meyer}
\address{Mathematisches Institut\\
  Georg-August Universit\"at G\"ottingen\\
  Bun\-sen\-stra\-\ss{}e 3--5\\
  37073 G\"ottingen\\
  Germany}
\email{rmeyer2@uni-goettingen.de}

\author{Georg Tamme}
\address{Universit\"at Regensburg\\
  Fakult\"at f\"ur Mathematik\\
  93040 Regensburg\\
  Germany}
\email{georg.tamme@ur.de}


\begin{abstract}
  Let~\(\dvr\) be a complete discrete valuation ring with residue
  field~\(\resf\) of positive characteristic and with fraction field~\(\dvf\) of characteristic~\(0\).
  We clarify the analysis behind the Monsky--Washnitzer completion
  of a commutative \(\dvr\)\nb-algebra using completions of bornological
  \(\dvr\)\nb-algebras.  This leads us to a functorial chain complex
  for a finitely generated commutative algebra over the residue field $\resf$ that computes its rigid cohomology in the sense of Berthelot.
\end{abstract}

\thanks{The first named author was supported by Conicet
and partially supported by grants UBACyT 20021030100481BA, PICT 2013-0454, and MTM2015-65764-C3-1-P (Feder funds).\\ The second named
  author was supported by DFG through CRC 878 and by the ERC through
  AdG 267\,079. \\
  The fourth named author was supported by DFG through CRC 1085.}

\maketitle
\begin{center}{Dedicated to Alain Connes on the occasion of his 70th birthday}\end{center}

\section{Introduction}
The problem of defining a cohomology theory with good properties for
an algebraic variety over a field~\(\resf\)
of non-zero characteristic has a long history.  In the breakthrough
paper~\cite{mw} by Monsky and Washnitzer, such a theory for smooth
affine varieties was constructed as follows.  Take a complete discrete
valuation ring~\(\dvr\) of mixed characteristic
with uniformizer~\(\dvgen\)
and residue field \(\resf=\dvr/\dvgen\dvr\)
(for example, \(\dvr\)
the ring of Witt vectors~\(W(\resf)\)
if~\(\resf\)
is perfect).  Let~\(\dvf\)
be the fraction field of~\(\dvr\).
Choose a \(\dvr\)\nb-algebra~\(R\)
which is a lift mod~\(\dvgen\)
of the coordinate ring of the variety and which is smooth
over~\(\dvr\)
(such a lift exists by~\cite{mme}).  Monsky--Washnitzer then introduce
the `weak' or dagger-completion~\(R^\updagger\)
of~\(R\)
and define their cohomology as the de~Rham cohomology of
\(R^\updagger\otimes_\dvr\dvf\).
The construction of a weak completion has become a basis for the
definition of cohomology theories in this context ever since.  The
Monsky--Washnitzer theory has been generalized by
Berthelot~\cite{berth} to ``rigid cohomology,'' which is a
satisfactory cohomology theory for general varieties over~\(\resf\).
Its definition uses certain de~Rham complexes on rigid analytic spaces.

The definition of the Monsky--Washnitzer cohomology as well as Berthelot's definition of
rigid cohomology depend on choices.
The chain complexes that compute them are not functorial for
algebra homomorphisms.  Only their homology is functorial.
Functorial complexes that compute
rigid cohomology have been constructed by Besser \cite{Besser}. However,
the construction is based on some abstract existence statements, and is not at all explicit.

In \cite{CCMT} we had developed a general framework for bornological structures on $\dvr$-algebras which allows in particular to generalize the weak completions of Monsky--Washnitzer to bornological versions of $J$-adic completions for an ideal $J$ in a $\dvr$-algebra. We had used this to study cyclic homology for such completions and to relate rigid cohomology to cyclic homology in this setting. As one major result we also had constructed a natural and explicit chain complex computing rigid cohomology for affine varieties over $\resf$. However \cite{CCMT} contained much more material than was needed to this latter end and the route to the construction of this complex was not the most direct possible. In the last section of \cite{CCMT} we had sketched a more direct approach. It is the aim of the present article to make this sketch explicit. At the same time we show that some of the basic arguments in \cite{CCMT} can be simplified significantly if one restricts the technical analysis to what is needed for that purpose. We obtain a conceptually simple construction of the natural complex computing rigid cohomology which is very much in the spirit of Grothendieck's infinitesimal cohomology for non-smooth varieties.

We proceed as follows. For a commutative $\dvr$-algebra $R$ we write $\ul{R}$ for $R\otimes \dvf$. We consider a presentation $J \to P \twoheadrightarrow R$ by a free commutative $\dvr$-algebra $P$. For the $m$-th powers of the kernel $J$, we obtain a projective system $(\comJ{P}{J^m})$ of bornological $J^m$-adic completions. Since the maps in this system are not surjective, the natural description of the de Rham cohomology of this pro-algebra is not via the de Rham complex of the projective limit but rather via the homotopy projective limit of the de Rham complexes (which is given as an explicit chain complex). As a consequence of homotopy invariance, this homotopy limit does not depend on the choice of a free presentation $P$ up to quasi-isomorphism. A functorial choice is $P= \dvr [A]$. One advantage of our approach is the fact that the formalism using bornological $J$-adic completions works fine also for algebras that are not Noetherian, such as $\dvr [A]$.

For the special choice of free presentation given by $\dvr [A]$, the homotopy limit of de Rham complexes mentioned above is completely explicit and manifestly functorial in $A$. Using results by Gro{\ss}e-Kl\"onne
in~\cite{gkdr} we show that it reproduces Berthelot's rigid cohomology.
\subsection{Notation}
\label{sec:notations}
Let~\(\dvr\) be a complete discrete valuation ring and
let~\(\dvgen\) be a generator for the maximal ideal in~\(\dvr\).
Let~\(\dvf\) be the fraction field of~\(\dvr\), that is,
\(\dvf=\dvr[\dvgen^{-1}]\).  Let \(\resf=\dvr/\dvgen \dvr\) be the
residue field. We will always assume that $\dvf$ has characteristic~\(0\).
Every element of~\(\dvf\) is written uniquely as \(x=u\dvgen^{\nu(x)}\),
where \(u\in \dvr\setminus\dvgen \dvr\) and \(\nu(x)\in\Z\cup\{-\infty\}\)
is the \emph{valuation} of~\(x\).  We fix \(0<\pepsilon<1\), and define
the \emph{absolute value} \(\abs{\hphantom{x}}\colon \dvf\to\R_{\ge
  0}\) by \(\abs{x}=\pepsilon^{\nu(x)}\) for \(x\neq 0\) and \(\abs{0}=0\).

If~\(M\) is a \(\dvr\)\nb-module, let~\(\ul{M}\) be the associated
\(\dvf\)\nb-vector space \(M \otimes \dvf\).  A \(\dvr\)\nb-module~\(M\) is
\emph{flat} if and only if the canonical map \(M \to \ul{M}\)
is injective, if and only if it is \emph{torsion-free}, that is,
\(\dvgen x=0\) implies \(x=0\).

Even though much of the foundational discussion works in greater generality, we will always assume that all algebras are commutative.
\section{Bornological modules over discrete valuation rings}
\label{sec:bornologies}

\begin{definition}
  Let~\(M\) be a \(\dvr\)\nb-module.  A (convex) \emph{bornology}
  on~\(M\) is a family~\(\bdd\) of subsets of~\(M\), called
  \emph{bounded} subsets, satisfying
  \begin{itemize}
  \item every finite subset is in~\(\bdd\);
  \item subsets and finite unions of sets in~\(\bdd\) are in~\(\bdd\);
  \item if \(S \in \bdd\), then the \(\dvr\)\nb-submodule generated
    by~\(S\) also belongs to~\(\bdd\).
  \end{itemize}
  A \emph{bornological $\dvr$\nb-module} is a $\dvr$\nb-module equipped with a bornology.\\
  A $\dvr$-linear map between bornological modules is said to be bounded if it maps bounded sets to bounded sets. A \emph{bornological \(\dvf\)\nb-vector space} is a bornological
  \(\dvr\)\nb-module such that multiplication by~\(\dvgen\) is an
  invertible map with bounded inverse.\\
  A bornological module equipped with an algebra structure is a bornological algebra, if the product of any two sets in $\bdd$ is again in $\bdd$ (i.e. if multiplication is bounded).
\end{definition}

\begin{definition}
  \label{def:converge}
  Let~\(X\)
  be a bornological \(\dvr\)\nb-module.
  Let~\((x_n)_{n\in\N}\)
  be a sequence in~\(X\)
  and let \(x\in X\).
  If \(S\subseteq X\)
  is bounded, then \((x_n)_{n\in\N}\)
  \emph{\(S\)\nb-converges}
  to~\(x\)
  if there is a sequence~\((\delta_n)_{n\in\N}\)
  in~\(\dvr\)
  with \(\lim \delta_n=0\)
  in the \(\dvgen\)\nb-adic
  topology and \(x_n-x \in \delta_n\cdot S\)
  for all \(n\in\N\).
  A sequence in~\(X\)
  \emph{converges} if it \(S\)\nb-converges
  for some bounded subset~\(S\)
  of~\(X\).
  If~\(S\)
  is bounded, then \((x_n)_{n\in\N}\)
  is \emph{\(S\)\nb-Cauchy}
  if there is a sequence~\((\delta_n)_{n\in\N}\)
  in~\(\dvr\)
  with \(\lim \delta_n=0\)
  and \(x_n-x_m \in \delta_l\cdot S\)
  for all \(n,m,l\in\N\)
  with \(n,m\ge l\).
  A sequence in~\(X\)
  is \emph{Cauchy} if it is \(S\)\nb-Cauchy
  for some bounded \(S\subseteq X\).
  The bornological \(\dvr\)\nb-module~\(X\)
  is \emph{separated} if limits of convergent sequences are unique.
  It is \emph{complete} if it is separated and for every bounded
  \(S\subseteq X\)
  there is a bounded \(S'\subseteq X\)
  so that all \(S\)\nb-Cauchy sequences are \(S'\)\nb-convergent.
\end{definition}
\begin{definition}
  \label{def:completion}
  The \emph{completion} of a bornological \(\dvr\)\nb-module~\(X\)
  is a complete bornological \(\dvr\)\nb-module~\(\comb{X}\)
  with a bounded map \(X\to \comb{X}\)
  that is universal in the sense that any map from~\(X\)
  to a complete bornological \(\dvr\)\nb-module
  factors uniquely through it.
\end{definition}
Completions of bornological \(\dvr\)\nb-modules
  always exist and may be constructed as follows.  Write
  \(X=\varinjlim {}(X_i)_{i\in I}\)
  as the inductive limit of the directed set of its bounded
  \(\dvr\)\nb-submodules.
  Then~\(\comb{X}\)
  is the separated quotient of the bornological inductive limit
  \(\varinjlim {}(\coma{X_i})_{i\in I}\) of the $\dvgen$-adic completions $\widehat{X_i}$. Note that a bornological $\dvr$-module $M$ is separated iff each of its bounded submodules is $\dvgen$-adically separated.

  A bounded \(\dvr\)\nb-linear map \(\varphi\colon X\to Y\) between bornological $\dvr$-modules is
  called a \emph{bornological quotient map} if any bounded
  subset~\(S\subseteq Y\) of~\(Y\) is \(\varphi(R)\) for some
  bounded subset~\(R\) of~\(X\).  Any bornological quotient map is surjective.
  \begin{lemma}
  \label{lem:completion_quotient}
  Taking completions preserves bornological quotient maps.
\end{lemma}
\begin{proof}
  Completions are obtained from the \(\dvgen\)\nb-adic
  completions of the bounded \(\dvr\)\nb-submodules.
  This reduces the assertion to the fact that \(\dvgen\)\nb-adic
  completions preserve surjections.
\end{proof}

\section{Bornological algebras and \textit{J}-adic completions}
\label{sec:properties}
Recall that all our algebras are commutative.
\begin{definition}\label{def:borcomp}
  Let~$R$ be a $\dvr$\nb-algebra and~$J$ an ideal in~$R$. We define the $J$-adic bornology on~$\ul{R}$
  as the bornology generated by subsets~$S$ of the form
  \bgl\label{def:Jadic}
  S = C\sum_{n\geq 1} \lambda_n M^n,
  \egl
  where $C\in \dvf$, $\lambda_n \in \dvf$, $M$ is a finitely generated $\dvr$\nb-submodule
  of~$J$ and there is $\alpha <1$ such that $\abs{\lambda_n}\leq \abs{\dvgen}^{-\alpha n}$ for all $n$.\\
  We denote by $\comJ{R}{J}$ the bornological completion of~$\ul{R}$
  with respect to this bornology.
\end{definition}
\begin{remark}\label{rsub} Since the modules described in \eqref{def:Jadic} are contained in each other, the bornology generated by them consists of subsets of sums $S+M$ of a module $S$ as in \eqref{def:Jadic} and of a finitely generated $\dvr$-module $M$. In Definition \ref{def:borcomp} we obtain the same bornology if we assume that the $\lambda_n$ are of the form $\lambda_n=\dvgen^{\beta (n)-n}$ for an increasing function $\beta: \N \to \N$ satisfying $\lim_n (\beta(n)/n)= \gamma$ for some $\gamma >0$ and such that $\beta (i+j) \geq\beta (i) + \beta (j)$ (one may take for $\beta (n)$ the largest integer $\lfloor \gamma n\rfloor$ less or equal to $\gamma n$). For such a choice of $\lambda_n$ we have $\abs{\lambda_{i+j}}\leq \abs{\lambda_{i}\lambda_j}$.\end{remark}
Since the product of two sets as in \eqref{def:Jadic} is contained in one of the same form, multiplication on~$\ul{R}$ is bounded for the $J$\nb-adic bornology and $\comJ{R}{J}$ is a complete bornological algebra.

Let~\(R\) be a finitely generated, commutative \(\dvr\)\nb-algebra.
Monsky--Washnitzer~\cite{mw} define the \emph{weak
  completion}~\(R^\updagger\) of~\(R\) as the subset of the
\(\dvgen\)\nb-adic completion
\[
\coma{R} = \varprojlim\limits_{\scriptstyle j} R/\dvgen^j R
\]
consisting of elements~\(z\) having representations
\begin{equation}\label{MW}
  z = \sum_{j=0}^\infty \dvgen^j w_j
\end{equation}
with \(w_j \in M^{\kappa_j}\),
where \(M\)
is a finitely generated \(\dvr\)\nb-submodule
of~\(R\) containing $1$
and \(\kappa_j \le c(j+1)\)
for some constant \(c>0\)
depending on~\(z\).

If $J=\dvgen R$, then the $J$-adic bornology and the corresponding completion $\comb{R_J}$ also make sense on $R$, rather than on $\ul{R}$, if we allow only $C=1$ in the definition of the basic bounded set $S$ in \ref{def:borcomp}. In fact, then necessarily $\lambda_nM^n\subseteq R$ in \eqref{def:Jadic} since $M^n$ is divisible by $\dvgen^n$.
\bprop\label{ppn}
Consider the ideals $I=\dvgen R$ and $I^m=\dvgen^m R$ in the $\dvr$-algebra $R$. The $I$-adic and the $I^m$-adic bornologies on $R$ are the same.
\eprop
\bproof The $I$-adic and the $I^m$-adic bornologies are generated, respectively, by sets $S$ and $S'$ of the form
\[
S = \sum_{n\geq 0}\dvgen^{\beta (n)}M^n \qquad S' = \sum_{n\geq 0}\dvgen^{\beta' (n)}\dvgen^{-n}\dvgen^{mn} (M')^n =\sum_{n\geq 0}\dvgen^{\beta' (n)+(m-1)n}(M')^n,
\]
where $M$ and~$M'$ are finitely generated $\dvr$-submodules in~$R$ containing~\(1\) and where we may assume that $\gamma n\leq \beta(n) \leq \gamma n+1$ and $\gamma'n-1\leq \beta'(n) \leq \gamma'n$ for suitable $\gamma, \gamma'>0$. It is clear that any set of the form~$S'$ is contained in one of the form~$S$ with \(M=M'\).

Conversely, for each $\ell \in \N$, a set of the form~$S$ can be decomposed as a finite sum over $i=0,\ldots , \ell -1$, of modules
\(M^i\sum_{n\geq 0}\dvgen^{\beta (n\ell+i)}M^{n\ell}\).
Since \(M^i\subseteq M^\ell\), each summand is contained in
\(\sum_{n\geq 1}\dvgen^{\beta ((n-1)\ell)}M^{n\ell}\).
We have $\beta ((n-1)\ell)/n\geq \gamma\cdot \ell - \gamma \ell/n$.  For \(n\ge2\) and sufficiently large~$\ell$, this is $\geq \gamma' + (m-1)$.  For such~$\ell$, we have $S\subseteq M^\ell + S'$ with \(M' = M^\ell\).
\eproof
\bprop\label{pdag} Assume that~$R$ is finitely generated and let $I=\dvgen R$. The natural map $\comb{R_I} \to \widehat{R}$ induces isomorphisms $\comb{R_I}\cong R^\updagger$ and
$\comJ{R}{I}\cong \underline{R^\updagger}$. \eprop
\begin{proof}
Since completion commutes with ${}\otimes \dvf$, it suffices to prove the first isomorphism. We do this first for the case of a free algebra. Thus let $P=\dvr [x_1,\ldots, x_n]$ and $I'=\dvgen P$. The $I'$-adic bornology on $P$ is generated by $\dvr$-submodules of the form
\[M^{(\dvgen)} =\sum_{k\geq 0}\dvgen^k M^{k+1}\]
where $M$ is a finitely generated $\dvr$-submodule of $P$. Since any such $M$ is contained in a power of the standard $\dvr$-submodule $L$ generated by $1,x_1,\ldots ,x_n$, we may assume that $M=L^j$ for some $j$. For $M_j=L^j$, the $\dvgen$-adic completion $\overline{M_j^{(\dvgen)}}$ of $M_j^{(\dvgen)}$ consists of all power series of the form $\sum_\alpha \lambda_\alpha x^\alpha$ with $\abs{\dvgen\lambda_\alpha }\leq \abs{\dvgen}^{\abs{\alpha}/j}$ and $\nicefrac{\abs{\lambda_\alpha}}{\abs{\dvgen}^{\abs{\alpha}/j}} \to 0$ for $\abs{\alpha} \to \infty$. In particular, this completion is contained in $\widehat{P}$. Now, the bornological $I'$-adic completion of $P$ is the separated direct limit, with respect to $j$, of the $\dvgen$-adic completions $\overline{M_j^{(\dvgen)}}$. Since the $\overline{M_j^{(\dvgen)}}$ are isomorphic to subsets of $\widehat{P}$ which are contained in each other, this is just the union and gives exactly the set of elements $z$ described in equation \eqref{MW}.

Now assume that $R=P/J$ for an ideal $J$ in $P$. Then it is already shown in \cite{mw} that $P^\updagger /JP^\updagger\cong R^\updagger $. The proof is as follows: \\
  The ring $P^\updagger$ is Noetherian by \cite{ful}, and convergence of the geometric series implies that the ideal $\dvgen P^\updagger$ is contained in the Jacobson radical. Then, using Krull's intersection theorem, the finitely generated $P^\updagger$-module $P^\updagger/JP^\updagger$ is $\dvgen$-adically separated. Since it is also a quotient of a weakly complete algebra, it is then isomorphic to its weak completion. Thus the natural map $P/J \to P^\updagger/JR^\updagger$ extends to a map $(P/J)^\updagger \to P^\updagger/JP^\updagger$, which is inverse to the natural map in the converse direction.

  Finally, note that, for any finitely generated submodule $M$ of $R$, the $\dvgen$-adic completion $\overline{M^{(\dvgen)}}$ is mapped to the set of elements described in \eqref{MW}, under the natural map $\comb{R_I} \to \widehat{R}$, and all such elements $z$ are in the union of such images. Thus $\comb{R_I}$ is mapped onto $R^\updagger$. Now we get maps
  \[P^\updagger\cong \comb{P_{I'}} \longrightarrow \comb{R_I} \longrightarrow R^\updagger\]
  The first map annihilates $JP^\updagger$. We thus get a factorization
  \[P^\updagger\cong \comb{P_{I'}} \longrightarrow P^\updagger/JP^\updagger\longrightarrow \comb{R_{I}} \longrightarrow R^\updagger\]
  Here, the composition of the first and second arrow is surjective by Lemma \ref{lem:completion_quotient} and the composition of the second and third arrow is an isomorphism by the above. Therefore the last arrow must be an isomorphism.
\end{proof}
\begin{remark}\label{ndag} It is immediate from the definitions that the natural map $\comb{R_I} \to \widehat{R}$ maps $\comb{R_I}$ onto $R^\updagger$. The detour via the free presentation $P$ of $R$ in the proof above was needed only in order to obtain injectivity.\\
\end{remark}

\begin{definition}
  \label{def:alpha-tube}
  Let \(J\) be an ideal in $R$ and assume that~$R$ is flat as a \(\dvr\)\nb-module.  The \emph{tube algebra} of~\(R\)
  around~\(J\) is
  \begin{equation}
    \label{eq:alpha-tube}
    \tub{R}{J}{} \defeq
    \sum_{n=0}^\infty \dvgen^{-n}J^n\subseteq \ul{R },
  \end{equation}
  where the \(0\)-th summand is \(J^0 \defeq R\).  This is a
  \(\dvr\)\nb-algebra.
\end{definition}
\begin{theorem}
  \label{pro:linear_growth_on_tube}
  Let \(J\triqui R\) be an ideal with \(\dvgen \in J\).  There is an isomorphism
  \[
  \ul{\tub{R}{J^m}{}}^\updagger
  \cong \comJ{R}{J^m}.
  \]
  for each $m\in \N$.
\end{theorem}
\bproof
In view of Proposition \ref{pdag} we have to show that the two bornologies on the algebra $\ul{\tub{R}{J^m}{}}=\ul{R}$ defining the $I$-adic, for $I=\dvgen \tub{R}{J^m}{}$, and the $J^m$-adic bornological completions are the same.

The first bornology is generated by sets of the form
  \bgl\label{bor1}S = S_{C,\beta,N,M_0,M}= C\sum_{n\geq 1}\dvgen^{\beta(n)} \left(M_0 + \sum_{i=1}^N \dvgen^{-i}M^{i}\right)^n\egl
  with $C\in K$, $\beta$ as in Remark \ref{rsub}, $N\in \N$, $M_0$ a finitely generated submodule of $R$ and $M$ a finitely generated submodule of $J^m$.
The second bornology is generated by modules of the form
  \bgl\label{bor2} S'= S'_{C,\beta,M} =C\sum_{n\geq 1}\dvgen^{\beta(n)-n}M^{n}
  \egl
  with $C\in K$, $\beta$ as in \ref{rsub} and $M$ a finitely generated submodule of $J^m$.
  If we set $M_0 =0$ and $N=1$ we get $S'_{C,\beta,M}=S_{C,\beta,1,0,M}$.  Thus $S'$ is bounded for the $I$-adic bornology.

  Conversely, consider a set $S$ as in \eqref{bor1}. We have
  \[\left(M_0 + \sum_{i=1}^N \dvgen^{-i}M^{i}\right)^n = \sum_{j+k=n}M_0^j\left( \sum_{i=1}^N \dvgen^{-i}M^{i}\right)^k\]
  Now, since $\abs{\dvgen^{\beta(j+k)}}\leq \abs{{\dvgen^{\beta(j)}}\dvgen^{\beta(k)}}$ it follows that
  \[\dvgen^{\beta(n)}\left(M_0 + \sum_{i=1}^N \dvgen^{-i}M^{i}\right)^n \subseteq \sum_{j+k=n}\left(\left(\dvgen^{\beta(j)}M_0^j\right)\dvgen^{\beta(k)}\left( \sum_{i=1}^N \dvgen^{-i}M^{i}\right)^k\right).\]
  Thus $S$ is contained in the product of the two submodules $\sum_{j\geq 0}\dvgen^{\beta(j)}M_0^j$ and $\sum_{k\geq 0}\dvgen^{\beta(k)}\left( \sum_{i=1}^N \dvgen^{-i}M^{i}\right)^k$; here we use the convention \(M_0^0=\dvr\cdot 1\).  Now
  \[\left( \sum_{i=1}^N \dvgen^{-i}M^{i}\right)^k \subseteq \sum_{i=1}^{Nk} \dvgen^{-i}M^{i},\]
  so that $\sum_{k\geq 0}\dvgen^{\beta(k)}\left( \sum_{i=1}^N \dvgen^{-i}M^{i}\right)^k$ is contained in a set of the form $\sum_{i\geq 0} \lambda_i\dvgen^{-i}M^{i}$ with $\abs{\lambda_i}\leq \abs{\dvgen}^{(\gamma/N) i}$ with $\gamma >0$. Thus it is bounded in the $J^{m}$-adic bornology.

 It remains to show that $\sum_{j\geq 1}\dvgen^{\beta(j)}M_0^j$ is bounded in the $J^m$-adic bornology. We have
  \[
\sum_{j\geq 1}\dvgen^{\beta(j)}M_0^j = \sum_{j\geq 1}\dvgen^{\beta(j)}\dvgen^{-j}(\dvgen M_0)^j.
\]
  Therefore, the sum is bounded in the $J_0$-adic bornology for $J_0=\dvgen R$.  By Proposition~\ref{ppn}, it is also bounded in the $J_0^m$-adic bornology. As $J_0^m\subseteq J^m$ it is then bounded in the $J^m$-adic bornology.
  \eproof

\section{De Rham cohomology for bornological pro-algebras}
\begin{definition}
  \label{def:dR}
  Let~$R$ be a commutative $\dvr$-algebra. Then
  \(\Omega_{\ul{R}}^*\) denotes the dg-algebra of K\"ahler
  differential forms for~\(\ul{R}\) (over the ground field~\(\dvf\))
  and \((\Omega_{\ul{R}}^*,d)\) is the de Rham complex
  for~\(\ul{R}\).  The de Rham complex for $\ul{R^\updagger}$ is
    \((\ul{R^\updagger}\otimes_{\ul{R}}\Omega_{\ul{R}}^*,d)\).  For
    an ideal $J$ in $R$, the de Rham complex
    for $\comJ{R}{J}$ is
    \((\comJ{R}{J}\otimes_{\ul{R}}\,\Omega_{\ul{R}}^*,d)\).
\end{definition}

\begin{remark}
  \label{rem:dR}
  For finitely generated~$R$,
  it is not difficult to show that
  $\underline{R}^\updagger\otimes_{\ul{R}}\Omega_{\ul{R}}^*$ and
  $\comJ{R}{J} \otimes_{\ul{R}}\Omega_{\ul{R}}^*$ are the
  bornological $\dvgen R$-adic and $J$-adic completions of $\Omega_{\ul{R}}^*$,
  respectively. This justifies our definitions of the de Rham
  complexes for $\underline{R}^\updagger$ and $\comJ{R}{J}$.
\end{remark}

\subsection{Homotopy invariance}
\label{sec:homotopy_invariance}
Let \(R\) and~\(S\)
be complete bornological \(\dvr\)\nb-\hspace{0pt}algebras
and let \(f_0,f_1\colon R\rightrightarrows S\)
be bounded unital algebra homomorphisms.
\begin{definition}A (dagger-continuous)
\emph{homotopy between \(f_0\)
  and~\(f_1\)}
is a bounded unital algebra homomorphism
\(F\colon R\to S\hotimes \dvr[x]^\updagger\)
with \((\id_S\hotimes \ev_t)\circ F = f_t\)
for \(t=0,1\). Here $\hotimes$ stands for the completed tensor product, i.e., the completion of the algebraic tensor product with respect to the bornology generated by tensor products of bounded sets.\end{definition}
\begin{lemma}
  \label{lem:HP_dagger_1-variable}
  The kernel of the chain map from
  $\left((\dvf\otimes\dvr[x]^\updagger)\otimes_{\dvf[x]}
    \Omega_{\dvf[x]}^*, d\right)$ to~$\dvf$ is contractible
  through a bounded chain homotopy.
\end{lemma}
\begin{proof}
We have $\Omega_{\dvf[x]}^1 =\dvf[x] d(x) \cong K\otimes \dvr [x]$ and $\Omega_{\dvf[x]}^j =0$ for $j\geq 2$. Recall that~$\dvf$ has characteristic~$0$. The absolute integration map
  \(i\colon \sum a_n x^n \mapsto \sum \frac{a_n}{n+1} x^{n+1}\)
  is a bounded linear map on~\(\dvf\otimes \dvr[x]^\updagger\).
  It satisfies \(d\circ i = \id\)
  and \(i\circ d = \id - P_0\),
  where~\(P_0\) is defined by \(P_0(x^n) = \delta_{n,0} x^n\).  Hence
  \(d\colon \dvf\otimes \dvr[x]^\updagger\to \dvf\otimes \dvr[x]^\updagger\) is homotopy
  equivalent to~\(\dvf\) concentrated in degree~\(0\).
\end{proof}

\begin{proposition}
  \label{pro:dagger-continuous_homotopy_HP}
  Assume that $R,S$ are finitely generated
    commutative $\dvr$-algebras, $J\lhd R$ and $I \lhd S$ are ideals, and let $P=\comJ{R}{J}$,
    $Q=\comJ{S}{I}$. Let \(f_0,f_1\colon P\rightrightarrows Q\) be bounded unital
  homomorphisms.  Assume that
  there is a dagger-continuous homotopy between
  \(f_0\) and~\(f_1\). Then the maps induced by $f_0, f_1$
    between the de Rham complexes
    \((\comJ{R}{J}\otimes_{\ul{R}}\,\Omega_{\ul{R}}^*,\,d)\)
    and
    \((\comJ{S}{I}\otimes_{\ul{S}}\,\Omega_{\ul{S}}^*,\,d)\)
    are homotopic with a bounded \(\dvf\)\nb-linear chain homotopy.

\end{proposition}

\begin{proof}
The de Rham complex for $\comJ{S}{I} \hotimes\dvr [x]^\updagger$ is $(\comJ{S}{I}\otimes_{\ul{S}}\Omega_{\ul{S}}^* ) \hotimes (\dvr [x]^\updagger \otimes_{\dvr[x]} \Omega_{\dvf[x]}^*)$. The assertion then follows from Lemma \ref{lem:HP_dagger_1-variable}.
\end{proof}

\subsection{Homotopy limit}
	\label{subsec:homotopy-limit}
Let~\(A\)
be a commutative \(\resf\)\nb-algebra.
We may view~\(A\)
as a \(\dvr\)\nb-algebra
with \(\dvgen\cdot A=0\).
Let \(S\subseteq A\)
be a set of generators for~\(A\).
Let~$P_S=\dvr[S]$ be the free commutative algebra on the set~$S$.
The inclusion map \(S\to A\)
defines a unital homomorphism \(p\colon P_S\to A\)
because~\(A\)
is commutative.  This is surjective because~\(S\)
generates~\(A\)
by assumption.  Let \(J\defeq \ker p\);
this contains~\(\dvgen\)
because \(\dvgen\cdot A=0\).
We shall be interested in the completions
\(\comJ{P_S}{J^m}\),
which are complete bornological \(\dvr\)\nb-algebras,
and in the chain complexes \((\comJ{P_S}{J^m}\otimes_{\ul{P_S}}\,\Omega_{\ul{P_S}}^*,d)\) that compute their de Rham cohomology. For the special choice of $S=A$ we denote these complexes by $\chaindR^{[m]}(A)$ and their homology by $\homdR_*^{[m]}(A)$.

Since $J^{m+1}\subseteq J^m$, the identity map from $P_S$ equipped with the $J^{m+1}$-adic bornology to $P_S$ equipped with the $J^m$-adic bornology is bounded. It induces a map $\comJ{P_S}{J^{m+1}}\to \comJ{P_S}{J^m}$, as well as a map between the associated de Rham complexes. Now, the natural definition for the homology of a pro-algebra, given by a projective system $(R_m)$ of algebras $R_m$, is the homology of the homotopy limit of the complexes computing the homology of the $R_m$.

Here, we have a projective system of
complexes defined by the maps
\[
\sigma_{m}\colon \; \chaindR^{[m+1]}(A)\to \chaindR^{[m]}(A).
\]
We write  \(\chaindR^\rig(A)\) for the homotopy limit of \(\{\chaindR^{[m]}(A)\}\) . Thus, by definition, \(\chaindR^\rig(A)\) is the mapping cone of the bounded
chain map
\begin{equation}
  \label{eq:hoprojlim_definition}
  1-\sigma:\prod_{m=1}^\infty \chaindR^{[m]}(A) \to
  \prod_{m=1}^\infty \chaindR^{[m]}(A),\qquad
  (x_m) \mapsto (x_m - \sigma_{m}(x_{m+1})),
\end{equation}
shifted by~$-1$. Hence there is an exact triangle
\begin{equation}\label{triang}
\xymatrix{\chaindR^{\rig}(A)\ar[r]&\prod_{m=1}^\infty \chaindR^{[m]}(A)\ar[r]^{1-\sigma}&\prod_{m=1}^\infty \chaindR^{[m]}(A)}.
\end{equation}
\begin{definition} The cohomology of \(\chaindR^\rig(A)\) is denoted by \(\homdR^\rig(A)\).\end{definition} Taking cohomology in \eqref{triang} we obtain
a short exact sequence
\[
0 \to \mathop{\lim\nolimits^1}_{m}
\homdR^{[m]}_{*-1}(A) \to
\homdR^\rig_*(A) \to
\lim_{m} \homdR^{[m]}_*(A) \to 0.
\]

\subsection{Different presentations}
\label{sec:compl}

The free presentation~\(\dvr[A]\) of~\(A\) used to
define~\(\chaindR^\rig(A)\) is natural but possibly very large.  For
computations, we want to use smaller generating sets.  Here we
  show that these give homotopy equivalent chain complexes.

\begin{proposition}
  \label{pro:free_presentation_independent}
  Let~\(A\) be a \(\resf\)\nb-algebra and let \(S\subseteq A\) be a
  generating set.  Let let \(J^S\triqui
  \dvr[S]\) be the kernel of the canonical homomorphism \(p^S\colon
  \dvr[S]\to A\).

  There are sequences of compatible bounded homomorphisms
  \[
  f\colon \comJ{\dvr[S]}{(J^S)^m} \to \comJ{\dvr[A]}{J^m},\qquad
  g\colon \comJ{\dvr[A]}{J^m} \to \comJ{\dvr[S]}{(J^S)^m}
  \]
  for $m\in\N$, such
  that $f\circ g$ and $g\circ f$ are homotopic to the identity
  through dagger-continuous homotopies that are compatible
  for \(m\in\N\).
\end{proposition}

\begin{proof}
  The inclusion map \(S\to A\) induces a unital homomorphism
  \(f\colon \dvr[S]\to\dvr[A]\).  It maps the kernel~\(J^S\) of
  \(p^S\colon \dvr[S]\to A\) into the kernel~\(J\) of \(p\colon
  \dvr[A]\to A\) because \(p\circ f=p^S\).  Hence it extends
  uniquely to a bounded unital algebra homomorphism \(f\colon
  \comJ{\dvr[S]}{(J^S)^m} \to
  \comJ{\dvr[A]}{J^m}\).  Since~\(S\)
  generates~\(A\), the homomorphism \(p^S\colon \dvr[S]\to A\) is
  surjective.  For each \(a\in A\), choose some \(g(a)\in \dvr[S]\)
  with \(p^S(g(a))=a\); we may assume \(g(a)=a\) for all \(a\in S\).
  These choices define a unital homomorphism \(g\colon
  \dvr[A]\to\dvr[S]\). 
  By construction, \(p^S\circ g(a)=p(a)\) for all \(a\in A\).  This
  implies \(p^S\circ g=p\) and hence \(g(J) \subseteq J^S\).
  Hence~\(g\) extends uniquely to a bounded unital algebra
  homomorphism \(g\colon \comJ{\dvr[A]}{J^m}
  \to \comJ{\dvr[S]}{(J^S)^m}\).  We have \(g\circ f =
  \id_{\dvr[S]}\) because \(g(a)=a\) for all \(a\in S\).

  The homomorphism \(f\circ g\colon \dvr[A]\to\dvr[A]\) is homotopic
  to the identity map through the homotopy \(H\colon
  \dvr[A]\to\dvr[A,t]\) defined by \(H(a) \defeq t\cdot a+ (1-t)
  fg(a)\) for all \(a\in A\).  Since \(p\circ fg = p\), \(H\) maps
  \(J\defeq \ker(p)\triqui \dvr[A]\) into \(J\otimes \dvr[t]\), the
  kernel of \(p\otimes \id_{\dvr[t]} \colon \dvr[A,t] =
  \dvr[A]\otimes\dvr[t] \to A\otimes \dvr[t]\).

  By the universal property of a completion,
  \(H\) extends
  uniquely to a bounded unital algebra homomorphism from $\comJ{\dvr[A]}{J^m}$ into the $J^m\otimes \dvr[t]$-adic completion of $\dvf[A]\otimes\dvr[t]$. But this completion is $\comJ{\dvr[A]}{J^m} \hotimes
  \dvr[t]^\updagger$. In fact, a standard bounded module of the form $S=\sum_n \dvgen^{\beta (n)-n}M^n$ with $M$ a finitely generated submodule of $J^m\otimes \dvr[t]$, is contained in the tensor product of modules $S_1=\sum_n \dvgen^{\beta_1 (n)-n}M_1^n$ and $S_2=\sum_n \dvgen^{\beta_2 (n)}M_2^n$ with $M_1, M_2$ finitely generated submodules of $J^m$ and $\dvr[t]$, respectively, and $\beta_1,\beta_2$ such that $\beta(n)=\beta_1(n)+\beta_2(n)$ with $\lim_n\beta_1(n)/n, \lim_n\beta_2(n)/n >0$. This shows that the $J^m\otimes\dvr[t]$-adic bornology is contained in the tensor product bornology. The converse inclusion of bornologies follows from the fact that $\dvr[A]\otimes\dvgen^m\dvr[t]\subseteq J^m\otimes \dvr[t]$ and Proposition \ref{ppn}.
  In conclusion~\(H\) is a
  dagger-continuous homotopy between \(f\circ g\) and the identity
  map.
\end{proof}

\begin{corollary}
  \label{cor:anypres}
  Let $A$ and~$S$ be as in
  Proposition~\textup{\ref{pro:free_presentation_independent}}.
  Then
  \(\chaindR^\rig(A)\) is naturally chain homotopy equivalent to the
  homotopy limit of the de Rham complexes for $\comJ{\dvr [S]}{(J^S)^m}$, for any generating set $S\subseteq A$.
\end{corollary}

\begin{proof}
  The
  chain homotopy equivalences between the de Rham complexes for $\comJ{\dvr [A]}{J^m}$ and $\comJ{\dvr [S]}{(J^S)^m}$ are compatible for different~$m$
  and hence induce a chain homotopy equivalence between the homotopy
  projective limits.
\end{proof}


\section{Rigid cohomology}
\label{sec:compare_rigid}
Let~\(A\) be a finitely generated, commutative \(\resf\)\nb-algebra.
We are going to identify \(\homdR^\rig(A)\) with Berthelot's
rigid cohomology of~\(A\), as our notation already suggests.
Berthelot defines rigid cohomology more generally for separated
\(\resf\)\nb-schemes of finite type in~\cite{Berthelot-finitude} using
the theory of rigid analytic spaces. His construction can be simplified
a little bit, if one uses Gro\ss{}e-Kl\"onne's theory of dagger spaces \cite{gkdagger}
instead.
In the following, we describe this simplified construction very roughly.
We refer to \cite{CCMT}*{\S6} for further details.


For each $n\in\N$, we consider the $\dvf$-algebra
\[
W_{n} \defeq \dvr[x_{1}, \ldots, x_{n}]^{\updagger} \otimes_{\dvr} \dvf.
\]
Using the description of the dagger completion in Proposition \ref{pdag}, this is a bornological algebra.  A \emph{dagger} $\dvf$\nb-algebra is a
quotient of some~$W_{n}$ by an ideal.


Given a dagger $\dvf$\nb-algebra $L$, we denote the set of its maximal
ideals by~$\Sp(L)$.  One declares certain subsets of $\Sp(L)$ to be admissible open,
 and one defines certain coverings of admissible open subsets by
admissible open subsets to be admissible. The admissible open subsets
with the admissible coverings then form a Grothendieck site. One also
constructs a structure sheaf $\cO_{\Sp(L)}$ on this site whose global sections
are given by $L$. An affinoid dagger space is by definition a pair of the form $(\Sp(L), \cO_{\Sp(L)})$.
By glueing affinoid dagger spaces, one obtains the general notion of a
dagger space. Particular examples of not necessarily affinoid dagger
spaces are given by admissible open subsets of an affinoid dagger space
with the induced structure sheaf.

For any dagger space $(X, \cO_{X})$ there is a notion of differential forms. There is a coherent sheaf
of $\cO_{X}$-modules $\Omega^{\dag,1}_{X}$ (where we suppress the
base field $\dvf$ from the notation) together with a $\dvf$\nb-linear derivation
$d\colon \cO_{X} \to \Omega^{\dag,1}_{X}$. The usual construction
yields a de Rham complex $\Omega^{\dag}_{X}$ of coherent sheaves on $X$, and the de
Rham cohomology of $X$ is defined as the hypercohomology $\Hy^*(X,\Omega^{\updagger}_{X})$.

\begin{remark}
    If~$R$ is a finitely generated
    $\dvr$\nb-algebra, then~$\ul{R^\updagger}$ is a dagger
    $\dvf$\nb-algebra. The complex of global sections of the de Rham complex
    $\Omega^{\dag}_{\Sp(\ul{R^\updagger} ) }$ of the affinoid dagger
    space $\Sp(\ul{R^\updagger})$ is isomorphic to the de Rham complex of $\ul{R^\updagger}$
    introduced in Definition~\ref{def:dR}.
    By \cite{gkdagger}*{Proposition~3.1} the higher cohomology of coherent sheaves on
    affinoid dagger spaces vanishes. This implies that the de Rham cohomology of the
    dagger space $\Sp(\ul{R^\updagger})$ is naturally isomorphic to the cohomology
    of the de Rham complex of $\ul{R^\updagger}$ from Definition~\ref{def:dR}.
\end{remark}

Let again~$R$ be a finitely generated $\dvr$\nb-algebra. As mentioned before,
$\ul{R^\updagger}$ is a dagger $\dvf$\nb-algebra, and we can consider the dagger space $\Sp(\ul{R^\updagger})$.
On the other hand, we also have the reduction $R/\dvgen R$ and its prime spectrum $\Spec(R/\dvgen R)$.
There exists a so called specialisation map
\[
\spec\colon \Sp(\ul{R^\updagger}) \to \Spec(R / \dvgen R).
\]
It turns out that, for any closed subset $Z \subset \Spec(R/\dvgen R)$,
the preimage
\[
]Z[ \defeq \spec^{-1}(Z) \subseteq \Sp(\ul{R^{\dag}})
\]
is an admissible open subset, hence a dagger space. It is called the tube of~$Z$ in $\Sp(\ul{R^{\dag}})$.
It is described more explicitly in \cite{CCMT}*{Rem.~6.2}.

We now consider a finitely generated $\resf$\nb-algebra $A$. The rigid cohomology of $A$ is constructed as follows.  We choose a smooth $\dvr$\nb-algebra $R$ together with a surjection $p\colon R \to A$.
Since $p$ factors over $R/\dvgen R$, it realizes $\Spec(A)$ as a closed subset of $\Spec(R /\dvgen R)$.
We thus have its tube $]\Spec(A)[$ in $\Sp(\ul{R^{\dag}})$.

\begin{definition}
  \label{def:rigid-cohomology}
  The \emph{rigid cohomology} of~$A$ with coefficients in~$\dvf$ is
  the de Rham cohomology of the tube~$]\Spec(A)[$, i.e.,
  \[
  H_\rig^*(A,\dvf) \defeq \Hy^*(]\Spec(A)[,\Omega^{\updagger}_{]\Spec(A)[}).
  \]
\end{definition}

By \cite{gkdr}*{Proposition~3.6}, the rigid cohomology groups
  defined here are canonically
  isomorphic to the groups defined by Berthelot in  \cite{Berthelot-finitude}*{(1.3.1)}. Up to isomorphism,
  they do not depend on the choice of $p\colon R\to A$.

We need a way to compute the rigid cohomology of $A$. In general, the tube $]\Spec(A)[$ is not affinoid.
However, using the tube algebras of Definition~\ref{def:alpha-tube}, we can construct an explicit admissible
covering by affinoid subsets as follows.
Choose a surjection $p\colon R \to A$ with a smooth $\dvr$-algebra $R$ as before and denote by $J$ the kernel of $p$.
Since $R$ is Noetherian, $J^{m}$ is a finitely generated ideal for any $m \in \N$. This implies that the tube algebra $\tub{R}{J^{m}}{}$ is finitely generated as a $\dvr$-algebra.
Hence $\ul{\tub{R}{J^{m}}{}^{\updagger}}$ is a dagger $\dvf$-algebra. The obvious map $\ul{R^{\dag}} \to \ul{\tub{R}{J^{m}}{}^{\updagger}}$ induces a map of dagger spaces
$\Sp( \ul{\tub{R}{J^{m}}{}^{\updagger}} ) \to \Sp( \ul{R^{\dag}} )$. Using the explicit construction of the tube algebra, it is not too hard to show that this map realizes $\Sp( \ul{\tub{R}{J^{m}}{}^{\updagger}} )$ as an admissible open subset of $]\Spec(A)[{} \subseteq \Sp( \ul{R^{\dag}} )$ and that we get  in fact an admissible covering
\begin{equation}
	\label{eq:admissible-covering-of-tube}
]\Spec(A)[{} = \bigcup_{m\geq 1} \Sp( \ul{\tub{R}{J^{m}}{}^{\updagger}} )
\end{equation}
(see \cite{CCMT}*{Rem.~6.2, Lemma~6.6}).

Our main result is the following.

\begin{theorem}
  \label{thm:rig}
  Let~\(A\) be a finitely generated, commutative
  \(\resf\)\nb-algebra.  There are natural isomorphisms
  \(\homdR_j^\rig(A) \cong H_\rig^{j}(A,\dvf)\).
\end{theorem}

\begin{proof}
We choose a particular smooth presentation of $A$ as in~\ref{subsec:homotopy-limit}.
Let $S \subseteq A$ be a finite set generating~$A$ as a $\resf$\nb-algebra. Then $R \defeq \dvr[S]$ is a smooth $\dvr$-algebra
which admits a canonical surjection $p\colon R \to A$. Let  \(J\triqui R\) be its kernel.
Corollary~\ref{cor:anypres} shows that
\(\chaindR^\rig(A)\) is homotopy equivalent to the homotopy limit
of the de Rham complexes \((\comJ{R}{J^m} \otimes
\Omega^*_{R},d)\) for \(m\in\N\).

We claim that the same complex computes the rigid cohomology of~$A$.
To see this, we use the admissible covering \eqref{eq:admissible-covering-of-tube} of the
tube $]\Spec(A)[{} \subseteq \Sp(\ul{R^{\dag}})$. By Theorem~\ref{pro:linear_growth_on_tube}, we may rewrite this using the bornological completions $\comJ{R}{J^{m}}$ as
\begin{equation}
	\label{eq:covering-of-tube-2}
 ]\Spec(A)[{} = \bigcup_{m\geq 1} \Sp( \comJ{R}{J^{m}} ).
\end{equation}
Recall that the rigid cohomology of $A$ is the de Rham cohomology of the dagger space $]\Spec(A)[$.
Since the higher cohomology of coherent sheaves on affinoid dagger spaces vanishes \cite{gkdagger}*{Proposition~3.1},
the de Rham cohomology of $\Sp( \comJ{R}{J^{m}} )$ is computed by the complex of  sections of the de Rham complex $\Omega^{\dag}_{\Sp( \ul{R^{\dag}})}$ over the admissible open subset
$\Sp( \comJ{R}{J^{m}} )$.  This complex is precisely \((\comJ{R}{J^m} \otimes \Omega^*_{R},d)\).
Finally, since  \eqref{eq:covering-of-tube-2} is an admissible covering by an increasing sequence of admissible open subsets,
a standard \v{C}ech cohomological argument shows that the de Rham cohomology of $]\Spec(A)[$ is computed by the homotopy
limit of the complexes \((\comJ{R}{J^m} \otimes \Omega^*_{R},d)\) for $m \in \N$.  This completes the proof.
\end{proof}

\end{document}